\documentclass{amsart}
\usepackage{amsmath}
\usepackage{amsxtra}
\usepackage{amscd}
\usepackage{amsthm}
\usepackage{amsfonts}
\usepackage{bbm}
\usepackage{cite}
\usepackage{graphicx}
\usepackage[bookmarksnumbered, colorlinks, plainpages]{hyperref}
\usepackage[utf8]{inputenc}
\usepackage{color}
\usepackage{enumerate}

\newcommand{\C}{\mathbb{C}}

\newtheorem{theorem}{Theorem}[section]

\newtheorem{lemma}[theorem]{Lemma}

\newtheorem{corollary}[theorem]{Corollary}
\newtheorem{proposition}[theorem]{Proposition}
\theoremstyle{definition}

\usepackage[paperwidth=17.10cm, paperheight=24.60cm, left=1.5cm, right=1.5cm, top=2cm, bottom=2cm]{geometry}

\usepackage{etoolbox}
\patchcmd{\section}{\scshape}{\large}{}{}
\patchcmd{\subsection}{\bfseries}{\normalfont}{}{}

\usepackage{titlesec}
\titleformat{\section}
{\normalfont\large\bfseries}
{\thesection}{1em}{}
\titleformat{\subsection}
{\normalfont\large\bfseries}
{\thesubsection}{1em}{}

\makeatletter
\def\@setauthors{%
	\begingroup
	\def\thanks{\protect\thanks@warning}%
	\trivlist
	\centering\footnotesize \@topsep30\p@\relax
	\advance\@topsep by -\baselineskip
	\item\relax\normalsize 
	\author@andify\authors
	\def\\{\protect\linebreak}%
	\authors%
	\ifx\@empty\contribs
	\else
	,\penalty-3 \space \@setcontribs
	\@closetoccontribs
	\fi
	\endtrivlist
	\endgroup
}
\def\@settitle{\begin{center}%
		\baselineskip14\p@\relax
		\bfseries\large
		\@title
	\end{center}%
}
\makeatother

\begin{document}
	\title[On images of polynomial maps]{Images of polynomial maps and the Ax-Grothendieck theorem over algebraically closed division rings}
	\author[Elad Paran]{Elad Paran\textsuperscript{1,\dag}}
	\author[Tran Nam Son]{Tran Nam Son\textsuperscript{2,3,\ddag,*}}
	
	\maketitle
	
	\begin{center}
		{\small 
			*Corresponding author\\
			\textsuperscript{1}Department of Mathematics and Computer Science,\\ The Open University of Israel\\
			\textsuperscript{2}Faculty of Mathematics and Computer Science,\\ University of Science,\\ Ho Chi Minh City, Vietnam \\
			\textsuperscript{3}Vietnam National University,\\ Ho Chi Minh City, Vietnam\\
			\textsuperscript{\dag}paran@openu.ac.il,\\
			\textsuperscript{\ddag}trannamson1999@gmail.com
		}
	\end{center}
	
	\begin{abstract} \sloppy We study the images of polynomial maps over algebraically closed division rings. Our first result generalizes the classical Ax-Grothendieck theorem: We show that if \( f_1, \ldots, f_m \) are elements of the free associative algebra \( D\langle X_1, \ldots, X_m \rangle \) generated by \( m \geq 1 \) variables over an algebraically closed division ring \( D \) of finite dimension over its center \( F \), and if the induced map \( f = (f_1, \ldots, f_m) \colon D^m \to D^m \) is injective, then \( f \) must be surjective. With no condition on the dimension over the center, our second result is that \( p(D) = D \) if \( p \) is either an element in \( F\langle X_1, \ldots, X_m \rangle \) with zero constant term such that \( p(F) \neq \{0\} \), or a nonconstant polynomial in $F[x]$. Furthermore, we also establish some Waring type results. For instance,  for any integer \( n > 1 \), we prove that every matrix in \( \mathrm{M}_n(D) \) can be expressed as  a difference of pairs of multiplicative commutators of elements from \( p(\mathrm{M}_n(D)) \), provided again that \( D \) is finite-dimensional over \( F \).
	\end{abstract}

	\section{Introduction}
	
	The classical Ax-Grothendieck theorem, proven in the 1960s independently by Ax \cite{Pa_Ax_68} and Grothendieck \cite{Pa_Grothendieck_66}, states that any injective polynomial map from the $m$-dimensional complex space into itself must be surjective. This result holds more generally over any algebraically closed field.  A striking noncommutative counterpart was established by Augat in \cite{Pa_Au_19}, yielding a free version of the theorem for  polynomial maps in freely non-commuting variables. The proof involves a detailed study of the noncommutative Jacobian matrix. Further progress was made in \cite{Pa_Bea_21}, where the authors considered polynomial evaluations on matrix algebras whose entries lie in an algebraically closed division ring in the sense of Makar-Limanov \cite{Pa_Ma_85}. One of the interesting results of \cite{Pa_Bea_21} states that for any nonconstant multilinear polynomial, the evaluation map on matrix algebra is surjective. Since injectivity of an evaluation map already implies that the corresponding polynomial must be nonconstant, these surjectivity results also apply in that context. Overall, these results suggest a common theme: Algebraic structures that share properties with algebraically closed fields, such as algebraically closed division rings, often admit similar surjectivity phenomena for polynomial evaluation maps.
	
	The notion of an algebraically closed field can be generalized to division rings in various ways. A natural approach is introduced in \cite[Chapter 5, $\S$16, Page 255]{Bo_Lam_01}: A division ring is said to be \textit{algebraically closed} if every nonconstant polynomial with one-sided coefficients in a single variable over the ring has a right root within the ring. In this work, we adopt this definition; It is well known (see \cite[Theorem 1]{Pa_Ni_41}  or \cite[(16.14) Theorem (Niven, Jacobson)]{Bo_Lam_01}) that the real quaternions algebra $\mathbb{H}$ is the first example of a noncommutative algebraically closed division ring.  For further details on the characterization of algebraically closed division rings, we refer the reader to \cite[pp.~255--256]{Bo_Lam_01}.  Other generalizations have been introduced by Makar-Limanov \cite{Pa_Ma_85}, Dauns \cite{Bo_Da_82}, and Smith \cite{Pa_Smi_77}.

	In this work, we establish a noncommutative analogue of the Ax-Grothendieck theorem:  For an arbitrary division ring $D$, let $D\langle X_1, \ldots, X_m \rangle$ denote the free associative algebra generated by $m$ variables over $D$.  Every element of this algebra induces a polynomial map from $D^n$ to $D$. We prove the following generalization of the Ax-Grothendieck theorem (Theorem~\ref{ax}):  If an algebraically closed division ring $D$ is of finite dimension over its center, and for $f_1, \ldots, f_m \in {D}\langle X_1, \ldots, X_m \rangle$, if the map  \begin{eqnarray*}
		f &=& (f_1, \ldots, f_m) \colon D^m \to D^m,\\
		&&(a_1,\ldots,a_m)\mapsto \left(f_1(a_1,\ldots,a_m),\ldots,f_m(a_1,\ldots,a_m)\right)
	\end{eqnarray*}
	is injective, then it is surjective. 
	
	In order to prove this result, we first note that by a theorem of Baer (see \cite[(16.15) Theorem (Baer)]{Bo_Lam_01}), if $D$ is not a field then it is a quaternion algebra over a real-closed field $R$. We then interpret $f$ as a polynomial map from $R^{4m}$ to $R^{4m}$ using a theorem of Wilczynski, and then apply a version of the Ax-Grothendieck theorem for real-closed fields, due to Bialynicki-Birula and Rosenlicht. 
	
	Furthermore, we take a moment to explore polynomials in central variables, with particular attention to the single-variable case, while the behavior in several variables remains an open question. As a direct consequence of the definition of algebraically closed division rings, we establish a straightforward analogue in this setting (see Proposition~\ref{m=1}). This naturally raises the question of how large the image of a polynomial can be when \( D \) is not assumed to be algebraically closed. To address this, we prove in Theorem~\ref{one} that the image of a non-constant polynomial over an infinite division ring \( D \) must itself be infinite, provided that \( D \) contains infinitely many conjugacy classes. Additionally, we discuss the existence of such division rings and noncommutative division rings with only finitely many conjugacy classes.

	Our second main result is Theorem~\ref{p(F) neq 0}, which states that if \( D \) is an algebraically closed division ring with center \( F \), and \( p \in F\langle X_1, \ldots, X_m \rangle \) is a polynomial with zero constant term satisfying \( p(F) \neq \{0\} \), then \( p(D) = D \). Building on this result, we show that the image of \( p \) contains all diagonalizable matrices. Furthermore, every matrix over \( D \) can be expressed as a product of two diagonalizable matrices. As a direct consequence, for any \( n > 1 \), every matrix in \( \mathrm{M}_n(D) \) can be written as a product of two elements from \( p(\mathrm{M}_n(D)) \) (see Theorems~\ref{theorem1} and \ref{theorem2}). In addition, when \( D \) is finite-dimensional over \( F \), we establish that every matrix in \( \mathrm{SL}_n(D) \) can be expressed as a product of two multiplicative commutators from \( p(\mathrm{M}_n(D)) \) (see Theorem~\ref{SL}). This leads to the further conclusion, via Theorem~\ref{The}, that every matrix in \( \mathrm{M}_n(D) \) can be written as a difference of pairs of multiplicative commutators of elements in \( p(\mathrm{M}_n(D)) \), provided that the characteristic of $D$ is different from $2$. Moreover, when \( D \) has characteristic \( 0 \), we show that any matrix \( A \in \mathrm{M}_n(D) \) can be represented as a product of two additive commutators of elements in \( p(\mathrm{M}_n(D)) \), except in the case where \( A \) is invertible and \( n \) is odd. In this setting, we also establish that any nilpotent matrix \( A \in \mathrm{M}_n(D) \) can be written as an additive commutator of elements in \( p(\mathrm{M}_n(D)) \), while any trace-zero matrix \( A \in \mathrm{M}_n(D) \) can be expressed as either a sum or a difference of two additive commutators. Finally, in the finite-dimensional case, leveraging results from \cite{Pa_Pa_23}, we provide an explicit characterization of \( p(\mathrm{M}_n(D)) \) (see Theorem~\ref{des}).

	For convenience, we introduce some notations. Let \( D \) be a division ring with center \( F \), and let \( n \) be a positive integer. We denote by \( \mathrm{M}_n(D) \) the ring of \( n \times n \) matrices over \( D \) and by \( \mathrm{GL}_n(D) \) the general linear group of invertible matrices in \( \mathrm{M}_n(D) \). The sets \( \mathrm{LT}_n(D) \) and \( \mathrm{UT}_n(D) \) refer to lower and upper triangular matrices in \( \mathrm{M}_n(D) \) with diagonal entries equal to $1$, respectively. For \( a_1, \ldots, a_n \in D \), the symbol \( \mathrm{diag}(a_1, \ldots, a_n) \) denotes the diagonal matrix with these entries. The identity matrix in \( \mathrm{M}_n(D) \) is \( \mathrm{I}_n \), and for any \( \lambda \in D \), the symbol \( \lambda \mathrm{I}_n = \mathrm{diag}(\lambda, \ldots, \lambda) \). Finally, \( \mathrm{SL}_n(D) \) represents the commutator subgroup of \( \mathrm{GL}_n(D) \), which coincides with the set of matrices in \( \mathrm{GL}_n(D) \) with a Dieudonné determinant of \( \overline{1} \), except when \( n = 2 \) and \( D \) is the field with two elements. For more details on the Dieudonné determinant, see~\cite{Pa_Li_21}.  
	
	This paper is organized as follows: In Section~\ref{Ax-Gro}, we prove the mentioned generalization of the Ax-Grothendieck theorem. Section~\ref{section al} is devoted to the images of noncommutative polynomials on algebraically closed division rings, where we discuss the multi-linear case, which is motivated by the well-known Lvov-Kaplansky conjecture. Lastly, Section~\ref{section ma} turns to matrix evaluations.

	\section{A noncommutative Ax-Grothendieck theorem}\label{Ax-Gro}
	
	\subsection{The single variable case}
	
	The classical Ax-Grothendieck theorem states that a polynomial map $$f = (f_1,\ldots,f_m) \colon \C^m \to \C^m$$ (here $f_1,\ldots,f_m \in \C[x_1,\ldots,x_m]$ and $m$ is a positive integer) which is injective is also surjective. The theorem holds over $\C$, and more generally, over any algebraically closed field. For the case of a single {\bf central} variable, one has the following trivial analogue for algebraically closed division rings:

	\begin{proposition}\label{m=1}
		Let $D$ be an algebraically closed division ring, and let $D[x]$ denote the ring of polynomials in a central variable  $x$ over $D$. If a polynomial $f \in D[x]$ is injective as a function from $D$ to $D$, then it is surjective.
	\end{proposition}
	
	\begin{proof}
		If $f$ is injective, then of course $f$ is not a constant. Let $a \in D$. Then $f-a$ is also not a constant, hence it has a zero $b \in D$, hence $f(b) = a$. Thus $f$ is surjective.
	\end{proof}

	Following the proof of Proposition~\ref{m=1}, one can see that the image of any nonconstant polynomial in \( D[x] \) must be all of \( D \). This naturally leads to the question of how large the image can be without assuming that \( D \) is algebraically closed. To explore this further, we provide the following answer, emphasizing division rings with infinitely many conjugacy classes. In this context, two elements \( a, b \in D \) are  said to be \textit{conjugate} if there exists \( g \in D \setminus \{0\} \) such that \( a = g b g^{-1} \).  As is customary, this defines an equivalence relation that partitions \( D \) into distinct conjugacy classes.
	
	\begin{theorem}\label{one}
		Let $D$ be an infinite division ring and let $p\in D[x]$ be a nonconstant polynomial. If $D$ contains infinitely many conjugacy classes, then the image of $p$ evaluated on $D$ must be infinite.
	\end{theorem}
	
	\begin{proof}
		Assume, for the sake of contradiction, that the image of \( p \) is finite.  Then, there must be some element \( a \) in the image for which we can find  elements \( b_1, b_2, \dots \in D \) lying in infinitely many distinct conjugacy  classes, all satisfying \( p(b_i) = a \). This, however, would mean that the  nonconstant polynomial \( p - a \) has roots in infinitely many conjugacy  
		classes, contradicting a classical result of Gordon and Motzkin  
		\cite[Theorem 2]{Pa_GoMo_65}, which states that at most \( t \) conjugacy  classes of \( D \) can contain roots of \( p - a \), where \( t \) is its degree.
	\end{proof}
	
	In Theorem~\ref{one}, it is essential to assume that \( D \) is infinite; otherwise,  if \( D \) is finite, it would necessarily be a finite field (see \cite[(13.1) Wedderburn's ``Little" Theorem]{Bo_Lam_01}),  implying that the image of \( p \) must also be finite. The existence of such a division ring $D$ is guaranteed, as division rings with infinite center  do exist, as stated in \cite[$\S4$, Example 3]{Pa_ChuLee_79} or \cite[Proposition 2.3.5]{Bo_Cohn_95}. Here, the infiniteness condition is essential to ensure the existence of infinitely many conjugacy classes. 
	
	As a concrete example, consider the skew Laurent ring \( D = K((x, \sigma)) \), where \( \sigma \) is an automorphism of infinite order of an infinite field \( K \). This ring has infinite dimension over its center and is not algebraically closed; Indeed, \( x \) has no square root in \( D \). With a slight adjustment, one can confirm that the center of \( D \) is infinite, which in turn implies that \( D \) contains infinitely many conjugacy classes.  
	
	Another notable example arises from Mal’cev–Neumann division rings constructed over fields using noncyclic free groups. These rings are also not algebraically closed, as no element of a free group can be expressed as a square. Moreover, free groups provide an abundance of elements that generate infinitely many conjugacy classes, specifically through powers of a chosen generator.  Both of these observations can be verified directly by considering the abelianization of the free group \( F_n = \langle y_1, \ldots, y_n \mid \cdot \rangle \), which leads to the natural quotient  
	\[
	F_n \to F_n / F_n' \cong \mathbb{Z}^n,
	\]  
	where each generator \( y_i \) maps to the standard basis vector \( e_i = (0, \dots, 0, 1, 0, \dots, 0) \) in \( \mathbb{Z}^n \). Here, \( F_n' \) denotes the commutator subgroup of \( F_n \).

	However, there do exist infinite division rings with only finitely many conjugacy  classes. For polynomials over such rings, we suspect that the image might be  finite, though we are not certain, leaving this as an open problem. The existence of infinite division rings with only finitely many conjugacy  classes is discussed in \cite{Pa_AnGu_24}, where it is noted that Cohn's  construction from \cite[Theorem 6.3]{Pa_Cohn_71} can be applied to obtain  an infinite division ring of characteristic \( p \), whose multiplicative  group has precisely \( p \) conjugacy classes.

	\subsection{The multivariable case}

	Consider a division ring $D$, and let $D \langle X_1,\ldots, X_m\rangle$ be the free algebra in $n$ variables over $D$. Via substitution, each element of this ring induces a {\it polynomial function} from $D^m$ to $D$. That is, a function that can be represented using only the variables, constants from $D$, addition, and multiplication. It should be noted that distinct elements of $D \langle X_1,\ldots,X_m\rangle$ may induce the same function, even for $m = 1$ and for an infinite $D$ (see for example in \cite[\S1]{AP19}). For such general polynomial functions, we have the following generalization of the Ax-Grothendieck theorem:
	\begin{theorem}[Ax-Grothendieck theorem over division rings]\label{ax}\sloppy
		Let $D$ be a centrally-finite algebraically closed division ring and let $m \geq 1$. Suppose that $f_1,\ldots,f_m \in D \langle X_1,\ldots,X_m\rangle $. If the map \begin{eqnarray*}
			f &=& (f_1, \ldots, f_m) \colon D^m \to D^m,\\
			&&(a_1,\ldots,a_m)\mapsto \left(f_1(a_1,\ldots,a_m),\ldots,f_m(a_1,\ldots,a_m)\right)
		\end{eqnarray*} is injective, then it is surjective. 
	\end{theorem}
	
	\begin{proof}
		\sloppy		If $D$ is a field, then the result is the classical Ax-Grothendieck theorem. Suppose that $D$ is not a field. Then, by a theorem of Baer \cite[(16.15) Theorem (Baer)]{Bo_Lam_01}, $D$ is a quaternion algebra over a real-closed field $R$, with standard basis $\{1,i,j,k\}$ such that $$i^2 = j^2 = k^2 = 1, \quad ij = -ji.$$ 
		
		We view $f_1,\ldots,f_m$ as functions from $R^{4m} \to R^4$ as follows: We represent each variable $X_\ell$ (with $1 \leq \ell \leq m$) as $X_\ell = y_{\ell,1}+y_{\ell,2}i+y_{\ell,3}j+y_{\ell,4}k$, where $y_{\ell,1},y_{\ell,2},y_{\ell,3},y_{\ell,4}$ represent variables taking values in $R$; For any $$\big((a_{11},a_{12},a_{13},a_{14}), \ldots,(a_{m1},a_{m2},a_{m3},a_{m4})) \in R^{4m},$$ we write the value of $$f_\ell\big(a_{11}+a_{12}i+a_{13}j+a_{14}k,\ldots,a_{m1}+a_{m2}i+a_{m3}j+a_{m4}k\big)$$
		as \begin{eqnarray*}
			&&f_{\ell,1}\big((a_{11},a_{12},a_{13},a_{14}), \ldots,(a_{m1},a_{m2},a_{m3},a_{m4}))\\
			&+&f_{\ell,2}\big((a_{11},a_{12},a_{13},a_{14}), \ldots,(a_{m1},a_{m2},a_{m3},a_{m4}))i\\
			&+&f_{\ell,3}\big((a_{11},a_{12},a_{13},a_{14}), \ldots,(a_{m1},a_{m2},a_{m3},a_{m4}))j\\
			&+&f_{\ell,4}\big((a_{11},a_{12},a_{13},a_{14}), \ldots,(a_{m1},a_{m2},a_{m3},a_{m4}))k,
		\end{eqnarray*} where $f_{\ell,1},f_{\ell,2},f_{\ell,3},f_{\ell,4}$ are functions from $R^{4m}$ to $R$. Thus the function $$f \colon D^m \to D^m$$ is injective (resp. surjective) if and only if the function $$f_R = \big((f_{1,1},f_{1,2},f_{1,3},f_{1,4},\ldots,f_{m,1},f_{m,2},f_{m,3},f_{m,4})\big)$$ from $R^{4m}$ to $R^{4m}$ is injective (resp. surjective). We now note the following non-obvious fact: Each of the functions $f_{\ell,1},f_{\ell,2},f_{\ell,3},f_{\ell,4}$ is itself a {\bf polynomial} in the ring $$R[y_{1,1},y_{2,2},y_{3,3},y_{4,4},\ldots,y_{m,1},y_{m,2},y_{m,3},y_{m,4}].$$ This is proven in \cite[Theorem 4.1]{Wil14} and also in \cite[Corollary 4]{AP19}. 
		
		We can now invoke a theorem of Bailynicki-Birula and Rosenlicht \cite{BBR62}, which states that the Ax-Grothendieck theorem holds for polynomials over $R$ (Bailynicki-Birula and Rosenlicht phrase their result in the case where $R$ is the field of real numbers, but their result holds for any real-closed field, see \cite[Theorem 11.4.2]{Bochnak}). Thus, if $f$ is injective, so is $f_R$, hence by the theorem of Bailynicki-Birula and Rosenlicht, $f_R$ is surjective, hence so is $f$. 
	\end{proof}
	
	Let us note the following open questions: Does the above theorem hold for an arbitrary algebraically closed division ring $D$? In particular, does it hold in the case where $D$ is the Makar-Limanov algebraically closed division ring?

	\section{Algebraically closed division rings}\label{section al}
	
	In this section, we describe the images of noncommutative polynomials evaluated on algebraically closed division rings. The result is stated as follows.
	
	\begin{theorem}\label{p(F) neq 0}
		Let $D$ be an algebraically closed division ring with center $F$ and let \( p \) be either an element in \( F\langle X_1, \ldots, X_m \rangle \) with zero constant term such that \( p(F) \neq \{0\} \), or a nonconstant polynomial in $F[x]$. Then, $p(D)=D$.
	\end{theorem}
	
	\begin{proof}
		The case where $p$ is a nonconstant polynomial in $F[x]$ can be treated similarly to the proof of Proposition~\ref{m=1}. We now turn to the remaining case. The proof is motivated from \cite{Pa_Wa_21}. Since \( p \) is a polynomial with zero constant term, it follows that \( p(0, \ldots, 0) = 0 \). Suppose \( p(F) \neq \{0\} \). Then, there exist \( a_1, \ldots, a_m \in F \) such that \( p(a_1, \ldots, a_m) \neq 0 \). This implies there is an index \( i \in \{1, \ldots, m\} \) for which the degree of \( x_i \) in \( p \) is nonzero; otherwise, all degrees of \( x_1, \ldots, x_m \) would be zero, leading to the contradiction \( 0 \neq p(a_1, \ldots, a_m) = p(0, \ldots, 0) = 0 \).
		
		Without loss of generality, assume the degree of \( x_1 \) in \( p \) is nonzero. Define \( f(x) = p(x, a_2, \ldots, a_m) \). Since \( f(a_1) = p(a_1, a_2, \ldots, a_m) \neq 0 \), it follows that \( f(x) \) is nonzero. For any \( c \in D \), since in an infinite division ring, there is no nonzero polynomial in $D[x]$ can vanish identically on $D$ (as established in \cite[(16.7) Theorem, pp. 251]{Bo_Lam_01}), it follows that the polynomial \( f(x) - c \) remains nonconstant in \( D[x] \). Since \( D \) is algebraically closed, Theorem~\ref{m=1} implies that \( f(D) = D \), which in turn shows that \( p(D) = D \). This completes the proof.
	\end{proof}
	
	Regarding Theorem~\ref{p(F) neq 0} in the multilinear case which is related to the famous Lvov-Kaplansky conjecture (see \cite{Pa_Ka_20}), if \( p \) is a multilinear polynomial, expressed as
	\[
	p = \sum_{\sigma \in S_m} \lambda_\sigma x_{\sigma(1)} \cdots x_{\sigma(m)},
	\]
	where each \( \lambda_\sigma \in F \) for \( \sigma \) in the symmetric group \( S_m \) of degree \( m \), and \( p(F) \neq 0 \), then it follows that \( \sum_{\sigma \in S_m} \lambda_\sigma \neq 0 \). Let \( \lambda = \sum_{\sigma \in S_m} \lambda_\sigma \). If \( A \) is an unital and associative algebra over \( F \), then for all \( a \in A \), we have $a = p(\lambda^{-1}a, 1, \ldots, 1),$
	meaning that \( p(A) = A \). Therefore, assuming \( p(F) \neq \{0\} \), it follows that not only \( p(D) = D \) but also \( p(A) = A \) for all \( F \)-algebras \( A \), rendering this case trivial in the given context. In the case where \( p(F) = \{0\} \), there is a deeper understanding of the situation, which will be addressed in a later section, drawing on ideas and results from \cite{Pa_Pa_23, Pa_Wa_21, Pa_Wa_22}.

	Theorem~\ref{p(F) neq 0} shows that there are two possible cases: either \( p(F) = \{0\} \) or \( p(D) = D \). This behavior also occurs in the context of the determinant map when \( D = F \). For an algebraically closed field \( F \), it is shown in \cite[Proposition 4.4]{Pa_BreVo_24} that if \( p \) is a nonconstant noncommutative polynomial over \( F \), then  $$\det(p(\mathrm{M}_n(F))) = F$$ for all positive integers \( n \). The proof relies on utilizing the universal division \( F \)-algebra of degree \( n \).  Motivated by Theorem~\ref{p(F) neq 0} and focusing on the determinant map, it is natural to consider the behavior of \( \det(p(\mathrm{M}_n(F))) = \{0\} \) for each positive integer \( n \).  Following a similar approach related to the universal division algebra, it is shown in \cite[Proposition 2.1]{Pa_KaVi_12} that   for any integer \( n > 1 \), if \( F \) is a field (not necessarily algebraically closed and possibly of arbitrary characteristic) and \( p \) is a noncommutative polynomial over \( F \), then the condition \(\det(p(\mathrm{M}_n(F))) = \{0\}\) leads to the conclusion that \( p = 0 \).
	
	It is unrealistic to expect that many cases exist where \( \det(p(A)) = 0 \) for some tuple \( A \) implies \( p(A) = 0 \). While counterexamples are easily found in the context of a single variable, the situation remains less clear for two variables. For instance, if $A = \begin{pmatrix} 0 & 1 \\ 0 & 0 \end{pmatrix}$ and \( p(x) = x \), then \( p(A) = 0 \), but \( A \neq 0 \). Thus, additional conditions are required, which are discussed in detail in Yaghoub Sharifi's blog. For the reader’s convenience, we provide a brief overview here. 
	
	First, if \( p(x) \in \mathbb{R}[x] \) has no real root and \( A \in \mathrm{M}_2(\mathbb{R}) \), then 
	$\det (p(A)) = 0$ implies $p(A) = 0$. Indeed, since \( p(x) \in \mathbb{R}[x] \), every root of \( p(x) \) must appear with its complex conjugate. Thus, we can express \( p(x) \) in the form
	\[
	p(x) = a \prod_{k=1}^m (x - z_k)(x - \overline{z}_k),
	\]
	where \( a \neq 0 \) and each \( z_k \in \mathbb{C} \setminus \mathbb{R} \) and $\overline{z_k}$ stands for the complex conjugate of $z_k$. This implies that 
	\[
	p(A) = a \prod_{k=1}^m (A - z_k \mathrm{I}_2)(A - \overline{z}_k \mathrm{I}_2), 
	\]
	and from this, we have \begin{eqnarray*}
		0 &=& \det (p(A)) \\
		&=& a^2 \prod_{k=1}^m \det(A - z_k \mathrm{I}_2) \det(A - \overline{z}_k \mathrm{I}_2) \\ 
		&=& a^2 \prod_{k=1}^m \det(A - z_k \mathrm{I}_2) \overline{\det(A - z_k \mathrm{I}_2)}.
	\end{eqnarray*}Thus, there must be some \( j \leq m \) for which \( \det(A - z_j I) = 0 \), meaning \( z_j \) is an eigenvalue of \( A \). Since \( A \in \mathrm{M}_2(\mathbb{R}) \), its other eigenvalue must be \( \overline{z}_j \), and because \( p(x) \) has no real roots, \( z_j \neq \overline{z}_j \). By the Cayley-Hamilton theorem, it follows that \( (A - z_j \mathrm{I}_2)(A - \overline{z}_j \mathrm{I}_2) = 0 \), and consequently, \( p(A) = 0 \).
	
	Furthermore, the implication that \(\det p(A) = 0\) leads to \(p(A) = 0\) does not hold  for \(A \in \mathrm{M}_2(\mathbb{C})\) or \(A \in \mathrm{M}_n(\mathbb{R})\) with \(n \geq 3\). To illustrate this, consider the polynomial \( p(x) = x^2 + 1 \in \mathbb{R}[x] \) and let \( A = i e_{11} \in \mathrm{M}_2(\mathbb{C}) \), where \( i = \sqrt{-1} \). Here, for positive integers \( i \) and \( j \), the matrix \( e_{ij} \) is defined such that its \((i,j)\)-entry is \( 1 \) and all other entries are \( 0 \).  We find that $p(A) = A^2 + \mathrm{I}_2 = e_{22},$
	which leads to \(\det(p(A)) = 0\) while \(p(A) \neq 0\). 
	
	Now, for \(n \geq 3\), take \(A = e_{1n} - e_{n1} \in \mathrm{M}_n(\mathbb{R})\). In this case, we have
	\[
	p(A) = A^2 + \mathrm{I}_n = (e_{1n} - e_{n1})^2 + \mathrm{I}_n = \mathrm{I}_n - e_{11} - e_{nn} = \sum_{k=2}^{n-1} e_{kk} \neq 0,
	\]
	yet it is clear that \(\det(p(A)) = 0\).
	
	Consider the scenario where \( p(x) \in \mathbb{R}[x] \) has a real root. In this case, for every integer \( n \geq 2 \), there are infinitely many matrices \( A \in \mathrm{M}_n(\mathbb{R}) \) for which \( \det(p(A))=0 \) while \( p(A) \neq 0 \). Indeed, let \( \alpha, \beta \in \mathbb{R} \) satisfy \( p(\alpha) = 0 \) and \( p(\beta) \neq 0 \). We can construct the block diagonal matrix $$A=\mathrm{diag}(\alpha,\underbrace{\beta,\ldots,\beta}_{n-1\text{ times}})\in\mathrm{M}_n(\mathbb{R}).$$ Then, we find that
	\[
	p(A) = \mathrm{diag}(p(\alpha),\underbrace{p(\beta),\ldots,p(\beta)}_{n-1\text{ times}}) = \mathrm{diag}(0,\underbrace{p(\beta),\ldots,p(\beta)}_{n-1\text{ times}}).
	\]
	Thus, it follows that \( \det(p(A)) = 0 \), yet \( p(A) \neq 0 \).
	
	\section{Matrix evaluations of noncommutative polynomials}\label{section ma}
	
	\subsection{Diagonalizable matrices}
	
	Following Theorem~\ref{p(F) neq 0} again, we know that $p(D) = D$ where \( p \) is either an element in \( F\langle X_1, \ldots, X_m \rangle \) with zero constant term such that \( p(F) \neq \{0\} \), or a nonconstant polynomial in $F[x]$. Hence, the image of $p$ evaluated on matrices over $D$ contains all diagonalizable matrices over $D$. 
	
	To see this claim, note that the image of $p$ is invariant under similarity transformations. Therefore, it suffices to show that the image of $p$ contains all diagonal matrices over $D$. Now, let $A = \mathrm{diag}(a_1, a_2, \ldots, a_n)$, where $n$ is a positive integer and $a_1, a_2, \ldots, a_n \in D$. Since $p(D) = D$, each $a_i$ can be expressed as $a_i = p(b_{1,i}, \ldots, b_{m,i})$ for some $b_{1,i}, \ldots, b_{m,i} \in D$. Setting $B_i = \mathrm{diag}(b_{i,1}, b_{i,2}, \ldots, b_{i,m})$ for $i \in \{1, 2, \ldots, n\}$. It then follows that $A = p(B_1, \ldots, B_m)$, as required.
	
	Therefore, we obtain the following corollary.
	
	\begin{corollary}\label{one variable}
		Let $D$ be an algebraically closed division ring with center $F$ and let $n$ be an integer greater than $1$. If \( p \) is either an element in \( F\langle X_1, \ldots, X_m \rangle \) with zero constant term such that \( p(F) \neq \{0\} \), or a nonconstant polynomial in $F[x]$,  then 	$p(\mathrm{M}_n(D))$ contains all diagonalizable matrices over $D$.
	\end{corollary}
	
	Now, we will show that, under certain restrictions, every matrix over an algebraically closed division ring can be expressed as a product of two images of polynomials evaluated on matrices over the base division ring. By Corollary~\ref{one variable}, it suffices to show that every matrix over a division ring can be written as a product of two diagonalizable matrices.
	
	Since this result is proven in \cite{Pa_DuSo_25}, but the work remains unpublished, we provide a brief proof here for completeness. First, for matrices over a field, this result holds by Botha's work \cite{Pa_Bo_98,Pa_Bo_99}. In the noncommutative case, we start by considering the infiniteness of the center of an algebraically closed division ring. 
	
	Let $D$ be an algebraically closed division ring with infinite center $F$. Since  $F$ is infinite, we can choose distinct elements $\lambda_1,\ldots,\lambda_n\in F$ and take $$B=\mathrm{diag}(\lambda_1,\ldots,\lambda_n)\in\mathrm{M}_n(F).$$ Now, let $A\in\mathrm{M}_n(D)$. The claim is trivial if $A$ has the form: $A=\lambda\mathrm{I}_n$ where $\lambda\in F$. Hence, we focus on the case where $A\notin\{\lambda\mathrm{I}_n\mid\lambda\in F\}$. First, we consider that $A\in\mathrm{GL}_n(D)$. By \cite[Theorem 2.1]{Pa_EgGo_19},  there exists $P\in \mathrm{GL}_n(D)$ such that $$P^{-1}AP=UHV$$ where $U\in \mathrm{LT}_n(D), V\in \mathrm{UT}_n(D)$ and $H=\mathrm{diag}(1,\dots,1,h)$ for some $h$ in $D\setminus\{0\}$. Combining the observation $$P^{-1}AP=(UB)(B^{-1}HV)$$ with \cite[Lemma 2.1]{Pa_DuSo_25_1}, we conclude that both $UB$ and $B^{-1}HV$ are diagonalizable matrices in $\mathrm{M}_n(D)$; hence  $A$ is a product of two diagonalizable matrices, as promised. Note that the global assumption in \cite[Lemma 2.1]{Pa_DuSo_25_1}, which requires all entries on the main diagonal to belong to \( F \), still holds in the case where there is exactly one element outside \( F \). The remaining case is when \( A \notin \mathrm{GL}_n(D) \). In this case, if \( A \) is nilpotent, we rely on the classical result that any nilpotent matrix over \( D \) is similar to its Jordan canonical form, which is indeed a matrix over \( F \) (see, for instance, \cite[Lemma 3.2]{Pa_AbLe_21}). Therefore, \( A \) is similar to a matrix over \( F \). This, combined with Botha’s results in \cite{Pa_Bo_98, Pa_Bo_99}, completes the proof. Our final focus is on the case where $A$ is not nilpotent. According to  \cite[Theorem 15, Page 28]{Bo_Ja_43}, there exists $P\in\mathrm{GL}_n(D)$ such that $$P^{-1}AP=\begin{pmatrix}
		A_1&0\\0&A_2
	\end{pmatrix}$$ where $A_1\in\mathrm{GL}_{n-t}(D)$ and $A_2\in\mathrm{M}_{t}(D)$ is nilpotent for some positive integer $t$. Following the arguments above, we conclude that  $A_1$ and $A_2$ can each be represented as products of two diagonalizable matrices in $\mathrm{M}_{n-t}(D)$ and $\mathrm{M}_t(D)$, respectively, implying $A\in\mathrm{M}_n(D)$ can also be represented similarly. These assertions complete the proof.	 
	
	With these arguments in place, we can derive the following theorem.
	
	\begin{theorem}\label{theorem1}
		Let $D$ be an algebraically closed division ring with infinite center $F$. If \( p \) is either an element in \( F\langle X_1, \ldots, X_m \rangle \) with zero constant term such that \( p(F) \neq \{0\} \), or a nonconstant polynomial in $F[x]$,  then every matrix over $D$ can be expressed as a product of two elements from $p(\mathrm{M}_n(D))$ 
	\end{theorem}
	
	\subsection{The center of an algebraically closed division ring}
	
	Theorem~\ref{theorem1} requires the infiniteness of the center of an algebraically closed division ring. In the case of finite dimensionality, we rely on Baer's theorem for algebraically closed finite-dimensional division rings (see, e.g., \cite[pp.~255--256]{Bo_Lam_01}). This theorem states that if \( D \) is an algebraically closed division ring that is finite-dimensional over its center, then one of the following assertions must hold:
	\begin{enumerate}[\rm (i)]
		\item \( D \) is an algebraically closed field.
		\item The center \( F \) of \( D \) is a real-closed field and \( D \) is the ordinary quaternion division ring.
	\end{enumerate}Recall that a field \( F \) is \textit{real-closed} if \( F \) is not algebraically closed but the field extension \( F(\sqrt{-1}) \) is algebraically closed. A division ring \( D \) with center \( F \) is called the \textit{ordinary quaternion division ring} if \( D \) has the form
	\[
	D = \{ a + bi + cj + dk \mid a, b, c, d \in F, \, i^2 = j^2 = k^2 = -1, \, ij = -ji = k \}.
	\]Based on our current understanding, it is not yet known whether the center of an algebraically closed division ring is infinite in the case of infinite dimensionality. On another note, Theorem~\ref{theorem1} holds for algebraic division rings. Recall that an \textit{algebraic} division ring is a division ring in which every element is a root of a nonzero polynomial over its center. 	
	
	Let $D$ be a division ring that is both algebraically closed and algebraic over its center $F$. It is well known that an algebraically closed field must be infinite, which completes the proof when $D$ is commutative. Now, suppose $D$ is noncommutative. By \cite[(13.11) Theorem (Jacobson)]{Bo_Lam_01}, if $F$ is finite, then $D$ must necessarily be an algebraically closed field, which leads to a contradiction. Therefore, $F$ must be infinite. 
	
	The above result leads to the following question: Does there exist an algebraically closed division ring \( D \) with finite center? At a first glance, this is likely a difficult question. However, Cohn’s construction provides a method for constructing such a division ring. We extend our gratitude to a user on Mathematics Stack Exchange (Question 5030291) for bringing this result to our attention. The construction proceeds as follows.  
	
	According to \cite{Pa_Co_73} and \cite[p.~308]{Bo_Cohn_95}, an \emph{existentially closed division ring} (or \emph{EC-division ring}) over a field \( k \) is a division ring \( D \) that is a \( k \)-algebra and satisfies the property that any existential sentence with constants from \( D \), which holds in some extension of \( D \), must already hold in \( D \) itself. For instance, if two matrices are similar over an extension of an EC-division ring \( D \), then they must also be similar over \( D \) itself.  
	
	A result \cite[Theorem 6.5.3]{Bo_Cohn_95} states that given a division ring \( K \) with a central subfield \( k \), there exists an EC-division ring \( D \) (over \( k \)) containing \( K \), in which every finite consistent system of equations over \( K \) has a solution. Here, a system is called \emph{consistent} if it has a solution in some extension division ring. Moreover, \cite[Corollary 6.5.6]{Bo_Cohn_95} establishes that the center of an EC-division ring \( D \) over \( k \) is precisely \( k \).  
	
	Furthermore, \cite[Theorem 8.5.1]{Bo_Cohn_95} asserts that any polynomial equation in one central variable with coefficients in a division ring has a solution in some extension of the division ring. To obtain an example where the center is finite, we take \( k \) to be a finite field, such as a prime subfield of \( K \) in positive characteristic. It is worth noting that this construction does not require the assumption that every finite consistent set of equations over \( K \) has a solution; rather, the essential ingredient is the existence of an EC-division ring over a finite field.  
	
	In summary, given a division ring \( K \) of positive characteristic, we can choose \( k \) as a finite field contained in its center. By Cohn’s book, there exists an EC-division ring \( D \) containing \( K \). Moreover, any polynomial equation of the form  
	\[
	a_mx^m+ a_{m-1}x^{m-1}+\dots+a_1x+ a_0 = 0,
	\]  
	where \( m \) is a positive integer and \( a_0, a_1, \dots, a_m \in D \) with \( a_m \neq 0 \), always has a solution in some extension of \( D \). Since \( D \) is an EC-division ring, it follows that such an equation has a solution in \( D \) itself. Finally, the center of \( D \) is precisely \( k \), providing the desired example.

	The case where \( D \) is a division ring that is not algebraic over its finite center is much more difficult than the algebraic case. To the best of our knowledge, no classification exists for division rings that are not algebraic over their finite centers. It is known, however, that such division rings must form infinite-dimensional vector spaces over their centers, and each element must commute with at least one element that is not algebraic over the center. Additionally, there exist noncommutative division rings with their finite centers (see \cite[$\S4$, Example 3]{Pa_ChuLee_79} and \cite[Proposition 2.3.5]{Bo_Cohn_95}).
	
	Returning to Theorem~\ref{theorem1}, we can also derive the following theorem:
	
	\begin{theorem}\label{theorem2}
		Let \( D \) be an algebraically closed division ring with center \( F \), except when \( D \) is a division ring that is not algebraic over its finite center $F$. If \( p \) is either an element in \( F\langle X_1, \ldots, X_m \rangle \) with zero constant term such that \( p(F) \neq \{0\} \), or a nonconstant polynomial in $F[x]$,  then every matrix over \( D \) can be written as a product of two elements from \( p(\mathrm{M}_n(D)) \).
	\end{theorem}
	
	\subsection{The finite-dimensional case}
	
	We conclude this section with the finite-dimensional case. A well-known theorem of Baer states that if \( D \) is a centrally-finite algebraically closed division ring, then one of the following assertions must hold:
	\begin{enumerate}[\rm (i)]
		\item \( D \) is an algebraically closed field.
		\item The center \( F \) of \( D \) is a real-closed field and \( D \) is the ordinary quaternion division ring.
	\end{enumerate}It is known that every matrix over an algebraically closed field $F$ or the ordinary quaternion division ring $D$ with center $F$ has a Jordan normal form over $F$ or $F(i)$, respectively (see \cite[Theorem B-3.66, Page 397]{Bo_Rot_15} and \cite[Theorem 5.5.3]{Bo_Ro_14}). Although \cite[Theorem 5.5.3]{Bo_Ro_14} is only stated for the field of real numbers, it also holds for real-closed fields. We may also follow another proof given in \cite[Proposition~3.7]{Pa_Fa_18}.
	
	Regarding Jordan normal forms, Panja and Prasad \cite{Pa_Pa_23} investigated the image of polynomials with zero constant term and Waring-type problems on upper triangular matrix algebras over an algebraically closed field. Their main result relies on several key concepts, which we briefly recall for clarity.
	
	Let $n\geq1$ be an integer and let $\mathrm{T}_n(F)$ denote the set of all $n \times n$ upper triangular matrices over $F$. The set of all strictly upper triangular matrices is denoted by $\mathrm{T}_n(F)^{(0)}$. More generally, for $t \geq 0$, the set of upper triangular matrices whose entries $(i, j)$ are zero whenever $j - i \leq t$ is denoted by $\mathrm{T}_n(F)^{(t)}$. Let $\mathcal{T}(\mathrm{T}_n(F))$ represent the set of all polynomial identities satisfied by $\mathrm{T}_n(F)$. It follows that:
	\[
	\mathcal{T}(F) \supset \mathcal{T}(\mathrm{T}_2(F)) \supset \mathcal{T}(\mathrm{T}_3(F)) \supset \cdots.
	\]
	Given a polynomial $p$ in noncommutative variables over $F$, its order is defined as the smallest integer $m$ such that $p \in \mathcal{T}(\mathrm{T}_m(F))$ but $p \notin \mathcal{T}(\mathrm{T}_{m+1}(F))$. Note that $\mathrm{T}_1(F) = F$. A polynomial $p$ has order $0$ if $p \in \mathcal{T}(F)$, and the order of $p$ is denoted by $\mathrm{ord}(p)$.

	\begin{theorem} {\rm \cite[Theorem 5.18]{Pa_Pa_23}} \label{lem:Pa}
		Let $n \geq 2$ be a positive integer and let $p \in F\langle X_1,\ldots,X_m \rangle$ be a polynomial with zero constant term in noncommutative variables over an algebraically closed field $F$.  Then, the following statements hold:
		\begin{enumerate}[\rm (i)]
			\item If $\mathrm{ord}(p) = 0$, then $p(\mathrm{T}_n(F))$ is a dense subset of $T_n(F)$ with respect to the Zariski topology.
			\item If $\mathrm{ord}(p) = 1$, then $p(\mathrm{T}_n(F)) = \mathrm{T}_n(F)^{(0)}$.
			\item If $1 < \mathrm{ord}(p) < n - 1$, then $p(\mathrm{T}_n(F)) \subseteq \mathrm{T}_n(F)^{(r-1)}$, and equality may not hold in general. 
			\item If $\mathrm{ord}(p) = n - 1$, then $p(\mathrm{T}_n(F)) = \mathrm{T}_n(F)^{(n-2)}$.
			\item If $\mathrm{ord}(p) \geq n$, then $p(\mathrm{T}_n(F)) = \{0\}$.
		\end{enumerate}
	\end{theorem}
	
	Consider a division ring $D$ with center $F$, and let $p \in F\langle X_1,\ldots,X_m \rangle$ be a polynomial with zero constant term in noncommutative variables over $F$. It follows that for any $A \in \mathrm{M}_n(D)$ and $P \in \mathrm{GL}_n(D)$, we have $p(P^{-1}AP) = P^{-1}p(A)P$. By combining this result with the well-know theorem of Baer and Theorem~\ref{lem:Pa}, we arrive at the following theorem.
	
	\begin{theorem}\label{des}
		Let $n \geq 2$ be an integer and let $D$ be a centrally-finite algebraically closed division ring with center $F$ and $$K=\begin{cases}
			F(i)\text{ if }D\text{ is noncommutative},\\
			F\text{ if }D\text{ is commutative}.
		\end{cases}$$ Consider $p \in F\langle x_1,\ldots,x_m \rangle$ as a polynomial in noncommutative variables with zero constant term over $K$. Then, the following statements holds:
		\begin{enumerate}[\rm (i)]
			\item If $0<\mathrm{ord}(p)<n$, then $$p(\mathrm{M}_n(D))\subseteq\{P^{-1}AP\mid P\in\mathrm{GL}_n(D), A\in\mathrm{T}_n(K)^{(r-1)}\}\subseteq\mathrm{M}_n(D).$$
			\item If $\mathrm{ord}(p)\geq n$, then $p(\mathrm{M}_n(D))=\{0\}$.
		\end{enumerate} 
	\end{theorem}
	
	Regarding the case where \( p \in F[x] \) is a nonconstant polynomial, obtaining concrete results can be challenging, especially for small values of \(n\) such as \(n=2\) or \(n=3\). For more details, see \cite{Pa_Lu_22, Pa_Me_23}.

	To conclude this section, we will prove that for any integer \(n > 1\), every matrix in \(\mathrm{SL}_n(D)\) can be expressed as a product of two multiplicative commutators from \(p(\mathrm{M}_n(D))\), based on the situation in Theorem~\ref{theorem2} and assuming further that \(D\) is a centrally-finite division ring. As mentioned earlier, one of the following statements holds:
	\begin{enumerate}[\rm (i)]
		\item \(D\) is an algebraically closed field.
		\item The center \(F\) of \(D\) is a real-closed field, and \(D\) is the ordinary quaternion division ring.
	\end{enumerate}
	
	\textbf{Case 1:} \(D\) is an algebraically closed field. According to \cite[Theorem 1]{Pa_SoDuHaBi_22}, every matrix in \(\mathrm{SL}_n(D)\) can be written as a product of two multiplicative commutators of involutions in \(\mathrm{GL}_n(D)\). When the characteristic of \(D\) is not 2, \cite[Proposition 2.3]{Pa_BiDuHaSo_23} shows that every involution in \(\mathrm{GL}_n(D)\) is diagonalizable and therefore lies in \(p(\mathrm{M}_n(D))\) by Corollary~\ref{one variable}. Hence, in this case, every matrix in \(\mathrm{SL}_n(D)\) can be expressed as a product of two multiplicative commutators of elements in \(p(\mathrm{M}_n(D))\).  
	
	\textbf{Case 2:} \(F\) is a real-closed field, and \(D\) is the ordinary quaternion division ring. By following Case 1 in the proof of \cite[Theorem 11]{Pa_HaNaSo_24} and \cite[Corollary 2.1]{Pa_BiLaMa_23} (or see \cite[Theorem 6]{Pa_BiNaTr_24}), every matrix in \(\mathrm{SL}_n(D) \setminus \{\lambda \mathrm{I}_n \mid \lambda \in F\}\) can be expressed as a product of two multiplicative commutators of involutions in \(\mathrm{GL}_n(D)\). Therefore, it suffices to consider matrices of the form \(\lambda \mathrm{I}_n\) with \(\lambda \in F\).  
	
	Let \(A = \lambda \mathrm{I}_n \in \mathrm{SL}_n(D)\) where \(\lambda \in F\). By \cite[Lemma 2.2]{Pa_BiLaMa_23}, it follows that \(\lambda^n = \pm 1\). If \(\lambda^n = 1\), then \(A \in \mathrm{SL}_n(F)\), and the proof follows from \cite[Theorem 1]{Pa_SoDuHaBi_22}. In the case where \(\lambda^n = -1\), it follows that \(n\) must be odd. If \(n = 2k\) is even, we get that \((\lambda^k)^2 = \lambda^n =  -1\), which contradicts the fact that \(F\) is real-closed. Therefore, when \(\lambda^n = -1\) and \(n\) is odd, we have \(A = \lambda \mathrm{I}_n = (-\mathrm{I}_n)(-\lambda \mathrm{I}_n)\). Since \(-\lambda \mathrm{I}_n \in \mathrm{SL}_n(F)\), by \cite[Theorem 1]{Pa_SoDuHaBi_22}, the matrix \(-\lambda \mathrm{I}_n\) can be expressed as a product of two multiplicative commutators:  
	\[
	-\lambda \mathrm{I}_n = A_1 A_2 A_1^{-1} A_2^{-1} A_3 A_4 A_3^{-1} A_4^{-1},
	\]  
	where \(A_1, A_2, A_3, A_4 \in \mathrm{GL}_n(F)\) and \(A_1^2 = A_2^2 = A_3^2 = A_4^2 = \mathrm{I}_n\). Furthermore,  
	\[
	-\mathrm{I}_n = (i \mathrm{I}_n)(j \mathrm{I}_n)(i \mathrm{I}_n)^{-1}(j \mathrm{I}_n)^{-1}.
	\]  
	Since \(i \mathrm{I}_n\) and \(j \mathrm{I}_n\) commute with all matrices in \(\mathrm{GL}_n(F)\),  
	\[
	(iA_1)(jA_2)(iA_1)^{-1}(jA_2)^{-1} = (i \mathrm{I}_n)(j \mathrm{I}_n)(i \mathrm{I}_n)^{-1}(j \mathrm{I}_n)^{-1} A_1 A_2 A_1^{-1} A_2^{-1}.
	\]  
	Note that \(iA_1, jA_2, A_3, A_4\) are diagonalizable and therefore belong to \(p(\mathrm{M}_n(D))\) by Corollary~\ref{one variable}. Thus, we conclude with the following theorem:
	
	\begin{theorem}\label{SL}
		\sloppy	Let \( D \) be a centrally finite algebraically closed division ring with center $F$ of characteristic different from $2$. If \( p \) is either an element in \( F\langle X_1, \ldots, X_m \rangle \) with zero constant term such that \( p(F) \neq \{0\} \), or a nonconstant polynomial in $F[x]$,  then every matrix in \(\mathrm{SL}_n(D)\) can be expressed as a product of two multiplicative commutators of  elements in \(p(\mathrm{M}_n(D))\).
	\end{theorem}
	
	We now turn our attention to arbitrary matrices, not necessarily in  \( \mathrm{SL}_n(D) \). With Theorem~\ref{SL} at our disposal, along with  some additional tools that will be introduced shortly, we will demonstrate  that under the conditions of Theorem~\ref{SL}, any matrix can be expressed either as a single additive commutator or as a difference of products of pairs of multiplicative commutators arising from the image of a noncommutative polynomial evaluated on $\mathrm{M}_n(D)$. To lay the groundwork for this result, we first establish the following two lemmas concerning idempotents. This approach is motivated by a similar observation for involutions, as every idempotent is similar to the matrix \(\mathrm{diag}(1, \ldots, 1, 0, \ldots, 0)\).

	\begin{lemma}\label{B C}
		Let \( D \) be a division ring, and let \( n > 1 \) be an integer.  
		For any matrix \( A \in \mathrm{M}_n(D) \), there exist matrices  
		\( B, C \in \mathrm{SL}_n(D) \) such that \( A \) can be written  
		as their difference, i.e., \( A = B - C \).
	\end{lemma}
	
	\begin{proof}
		Referring to the proof of \cite[Lemma 1]{Pa_Va_05}, we see that for any matrix \( A \in \mathrm{M}_n(D) \), there exist invertible matrices \( P, Q \in \mathrm{GL}_n(D) \) such that \( PAQ \) has all diagonal entries equal to zero. Additionally, by \cite[Lemma 2.3]{Pa_Li_21}, any invertible matrix \( M \in \mathrm{GL}_n(D) \) can be factored as either \( M_1 M_2 \) or \( M_2 M_1 \), where \( M_1 \in \mathrm{SL}_n(D) \) and \( M_2 \) is a diagonal matrix of the form \( \mathrm{diag}(1,1,\dots,1,\alpha) \) with \( \alpha \in D \setminus \{0\} \). Applying this result to \( P \) and \( Q \), we express them as \( P = P_2 P_1 \) and \( Q = Q_1 Q_2 \), where \( P_1, Q_1 \in \mathrm{SL}_n(D) \) and \( P_2, Q_2 \) are diagonal matrices of the form \( \mathrm{diag}(1,1,\dots,1,p_2) \) and \( \mathrm{diag}(1,1,\dots,1,q_2) \), respectively, with \( p_2, q_2 \in D \setminus \{0\} \).
		
		Now, setting \( N = PAQ \), we obtain 
		\[
		P_1 A Q_1 = P_2^{-1} N Q_2^{-1},
		\]
		which also has all diagonal entries equal to zero. Since matrices of this form can always be decomposed as the difference of a unit lower triangular matrix and a unit upper triangular matrix, we write 
		\[
		P_1 A Q_1 = B_1 - C_1,
		\]
		where \( B_1 \in \mathrm{LT}_n(D) \) and \( C_1 \in \mathrm{UT}_n(D) \). Finally, letting 
		\[
		B = P_1^{-1} B_1 Q_1^{-1}, \quad C = P_1^{-1} C_1 Q_1^{-1},
		\]
		we conclude that $A = B - C,$ completing the proof.
	\end{proof}
	
	\begin{lemma}\label{C}
		Let $D$ be a centrally finite algebraically closed division ring of characteristic $0$, and let \( n > 1 \) be an integer.  The following statements hold:  
		\begin{enumerate}[\rm (i)]  
			\item Any nilpotent matrix \( A \in \mathrm{M}_n(D) \) can be expressed as an additive commutator of idempotents in \( \mathrm{M}_n(D) \).  
			\item If \( A \in \mathrm{M}_n(D) \) has trace zero, then it can be written as the sum or the difference of two additive commutators of idempotents in \( \mathrm{M}_n(D) \).  
			\item Any matrix \( A \in \mathrm{M}_n(D) \) can be represented as a product of two additive commutators of idempotents in \( \mathrm{M}_n(D) \), except when \( A \) is invertible and \( n \) is odd.  
		\end{enumerate}  
	\end{lemma}
	
	\begin{proof}
		According to Baer’s theorem, if $D$ is an algebraically closed division ring that is finite-dimensional over its center, then it must fall into one of the following two categories:
		\begin{enumerate}[\rm (a)]
			\item \( D \) is an algebraically closed field.
			\item The center \( F \) of \( D \) is a real-closed field and \( D \) is the ordinary quaternion division ring.
		\end{enumerate}	For the first case (a), item (i) is justified by \cite[Proposition 6]{Pa_Dr_02}, whereas items (ii) and (iii) follow directly from \cite[Propositions 7.3 and 7.5]{Pa_Ma_23}.
		
		Moving on to the second case (b), we observe that \( D \) contains a subfield \( K \), where \( K \) is an extension of \( F \) generated by \( i \). Following the proof of \cite[Proposition~3.7]{Pa_Fa_18}, we know that any \( A\in\mathrm{M}_n(D) \) has a Jordan normal form, meaning that there exists \( P \in \mathrm{GL}_n(D) \) such that  
		\[
		P^{-1} A P = \bigoplus_{i=1}^s J_{m_i}(\alpha_i),  
		\]  
		for some positive integers \( s, m_1, m_2, \ldots, m_s \), and scalars \( \alpha_1, \alpha_2, \ldots, \alpha_s \in K \), where \( m_1 + m_2 + \cdots + m_s = n \). Here,  
		\[
		J_{m_i}(\alpha_i) = \begin{pmatrix} 
			\alpha_i & 1 & 0 & \cdots & 0 & 0 & 0 \\ 
			0 & \alpha_i & 1 & \cdots & 0 & 0 & 0 \\ 
			0 & 0 & \alpha_i & \cdots & 0 & 0 & 0 \\ 
			\vdots & \vdots & \vdots & \ddots & \vdots & \vdots & \vdots \\ 
			0 & 0 & 0 & \cdots & \alpha_i & 1 & 0 \\ 
			0 & 0 & 0 & \cdots & 0 & \alpha_i & 1 \\ 
			0 & 0 & 0 & \cdots & 0 & 0 & \alpha_i
		\end{pmatrix} \in \mathrm{M}_{m_i}(D).
		\]  As \( \bigoplus_{i=1}^s J_{m_i}(\alpha_i) \) belongs to \( \mathrm{M}_n(K) \) and the expression of a matrix as a product of additive commutators of idempotents is preserved under similarity, assertion (iii) follows seamlessly from the first case (a). If $A$ is nilpotent, then all scalars $\alpha_1, \alpha_2, \ldots, \alpha_s$ must be zero, thereby establishing assertion (i). 
		
		We now turn our attention to the case of \( A \in \mathrm{M}_n(D) \) with trace zero. As established in \cite[Proposition 1.8]{Pa_AmRo_21} by induction on \( n \), if \( A \notin \{ \lambda \mathrm{I}_n \mid \lambda \in F \} \), then there exists \( P \in \mathrm{GL}_n(D) \) such that \( P^{-1} A P \) has all diagonal entries equal to zero. Notably, when \( D \) has characteristic zero, the set \( \{ \lambda \mathrm{I}_n \mid \lambda \in F \} \) is empty, ensuring that \( A \) is always similar to some matrix \( C \) with zero diagonal entries. At this time, it suffices to consider the matrix \( C \). We decompose \( C \) as \( C = U + L \), where \( U = [u_{ij}] \) is strictly upper triangular, given by  
		\[
		u_{ij} = 
		\begin{cases} 
			t_{ij}, & \text{if } i \leq j, \\  
			0, & \text{if } i > j,  
		\end{cases}
		\]
		and \( L = C - U \) is lower triangular. Clearly, both \( U \) and \( L \) are nilpotent. Hence, this completes the proof for the sum representation. The remaining part, concerning the difference representation, is carried out in a similar manner.
	\end{proof}
	
	Recall  that  every idempotent matrix over a division ring is similar to the matrix $$\mathrm{diag}(1, \ldots, 1, 0, \ldots, 0).$$ Moreover, with Lemmas~\ref{B C},~\ref{C}, and Corollary~\ref{one variable} in place, we now have everything needed to establish the following theorem.

	\begin{theorem}\label{The}
		\sloppy Let \( D \) be a centrally finite algebraically closed division ring with center \( F \), where \( \mathrm{char}(F) \neq 2 \). Suppose that \( p \) is either a nonconstant polynomial in \( F[x] \) or an element of \( F\langle X_1, \dots, X_m \rangle \) with zero constant term, satisfying \( p(F) \neq \{0\} \). Then, every matrix in \( \mathrm{M}_n(D) \) can be expressed as a difference of pairs of multiplicative commutators of elements from \( p(\mathrm{M}_n(D)) \). Moreover, if \( F \) has characteristic \( 0 \), the following hold:
		\begin{enumerate}[\rm (i)]
			\item Any nilpotent matrix \( A \in \mathrm{M}_n(D) \) can be written as an additive commutator of elements in \( p(\mathrm{M}_n(D)) \).
			\item If \( A \in \mathrm{M}_n(D) \) has trace zero, then it can be expressed as either a sum or a difference of two additive commutators of elements in \( p(\mathrm{M}_n(D)) \).
			\item Any matrix \( A \in \mathrm{M}_n(D) \) can be written as a product of two additive commutators of elements in \( p(\mathrm{M}_n(D)) \), except in the case where \( A \) is invertible and \( n \) is odd.
		\end{enumerate}
	\end{theorem}

	\section*{Declarations}
	
	Our statements here are the following:
	
	\begin{itemize}
		\item {\bf Ethical Declarations and Approval:} The authors have no any competing interest to declare that are relevant to the content of this article.
		\item {\bf Competing Interests:} The authors declare no any conflict of interest.
		\item  {\bf Authors' Contributions:} All two listed authors worked and contributed to the paper equally. The final editing was done by the corresponding author Tran Nam Son and was approved by all of the present authors.
		\item {\bf Availability of Data and Materials:} Data sharing not applicable to this article as no data-sets or any other materials were generated or analyzed during the current study.
	\end{itemize}

	\bibliographystyle{amsplain}

\end{document}